\documentclass{amsart}
\usepackage{amsmath,amssymb,amsthm}
\usepackage{paralist}
\usepackage{booktabs}
\usepackage[usenames,dvipsnames]{color}
\usepackage{comment}
\usepackage{graphicx}
\usepackage{latexsym}
\usepackage[latin1]{inputenc}
\usepackage{mathtools}
\usepackage[shortlabels]{enumitem}
\usepackage{thmtools}
\usepackage{url}
\usepackage[all,cmtip]{xy}

\usepackage{hyperref} 

\usepackage[enableskew]{youngtab}
\usepackage{ytableau}

\def\Ups{{\small {\Upsilon}}}

\def\nn{\mathbb N}

\def\ga{\gamma}

\def\al{\alpha}

\def\CR{\mathcal L}

\def\<{\langle}
\def\>{\rangle}
\def\La{\Lambda}

\def\RR{\text{\sc R}}
\def\rH{ {\text {\rm H}  } }

\def\di{\diamond}

\def\Ups{\Upsilon}

\def\lra{\longrightarrow}

\def\.{\hskip.06cm}
\def\ts{\hskip.03cm}

\def\fr{f}

\title[Asymptotics of principal evaluations of Schubert polynomials]{Asymptotics of principal evaluations
of Schubert polynomials for layered permutations}

\author[Alejandro Morales, Igor Pak, Greta Panova]{Alejandro H.~Morales$^\star$,
\ \ Igor Pak$^\di$, \ and \ \ Greta Panova$^\dagger$}

\thanks{\today}

\thanks{\thinspace ${\hspace{-.45ex}}^\star$Department of Mathematics
  and Statistics,
UMass, Amherst, MA.
\hskip.06cm
Email:
\hskip.06cm
\texttt{ahmorales@math.umass.edu}}

\thanks{\thinspace ${\hspace{-.45ex}}^\di$Department of Mathematics,
UCLA, Los Angeles, CA.
\hskip.06cm
Email:
\hskip.06cm
\texttt{pak@math.ucla.edu}}

\thanks{\thinspace ${\hspace{-.45ex}}^\dagger$Institute of Advanced Studies,
Princeton, NJ;
Department of Mathematics,
 UPenn, Philadelphia, PA}

\thanks{\hskip.1cm
Email:
\hskip.06cm
\texttt{panova@math.upenn.edu}}

\definecolor{darkgreen}{rgb}{0,0.7,0}
\definecolor{purplish}{rgb}{0.5,0,0.8}

\hypersetup{
  breaklinks,colorlinks,citecolor=darkgreen,linkcolor=purplish, 
  pdftitle={Notes: hook length formula skew shapes}
}

\declaretheorem[numberwithin=section]{theorem}
\declaretheorem[numberlike=theorem]{lemma}

\declaretheorem[numberlike=theorem]{proposition}

\declaretheorem[numberlike=theorem]{conjecture}

\declaretheorem[numberlike=theorem, style=definition]{remark}

\numberwithin{equation}{section} 

\begin{document}

\maketitle

\begin{abstract}
Denote by $u(n)$ the largest principal specialization of the Schubert polynomial:
$$
u(n) \, := \, \max_{w \in S_n} \. \mathfrak{S}_w(1,\ldots,1)
$$
Stanley conjectured in~\cite{St} that there is a limit
$$\lim_{n\to \infty} \, \frac{1}{n^2} \. \log u(n),
$$
and asked for a limiting description of permutations achieving
the maximum~$u(n)$.  Merzon and Smirnov conjectured in~\cite{MeS}
that this maximum is achieved on layered permutations.  We resolve
both Stanley's problems restricted to layered permutations.
\end{abstract}

\vskip.6cm

\section{Introduction}

Understanding the large-scale behavior of combinatorial objects is
so fundamental to modern combinatorics, that it has become routine and
no longer requires justification.  However, in \emph{algebraic combinatorics},
there are fewer results in this direction, as the objects tend to be have
more structure and thus less approachable.  This paper studies the
asymptotic behavior of the
principal evaluation of Schubert polynomials, partially resolving
an open problem by Stanley~\cite{St}. As the reader shall see,
the results are surprisingly precise.

\smallskip

\subsection*{Main results}
\emph{Schubert polynomials} \ts
$\mathfrak{S}_w(x_1,\ldots,x_n) \in \nn[x_1,\ldots,x_n]$, \ts $w\in S_n$,
were introduced by Lascoux and Sch\"utzenberger \cite{LSch} to study
Schubert varieties. They have been intensely studied
in the last two decades and remain a central object in algebraic combinatorics.
 The principle evaluation of the Schubert polynomials
can be defined via \emph{Macdonald's identity} \cite[Eq.~6.11]{Mac}:
\begin{equation} \label{eq:macdonald}
\Ups_w \, := \, \mathfrak{S}_w(1,\ldots,1) \, = \,
\frac{1}{\ell!} \. \sum_{(a_1,\ldots,a_{\ell}) \in \RR(w)} \. a_1 \ts \cdots \ts a_{\ell}\..
\end{equation}
Here $\ell=\ell(w)$ is the \emph{length} of $w$ (the number of inversions, and
\ts $\RR(w)$ \ts denotes the set of \emph{reduced words} of $w\in S_n\ts$: \ts tuples
\ts $(a_1,\ldots,a_{\ell})$ \ts  such that \ts $s_{a_1}\cdots
s_{a_{\ell}}$ \ts is a reduced decomposition of $w$ into simple
transpositions $s_i = (i,i+1)$.

Note that $\Ups_w$ has a more direct (but less symmetric) combinatorial
interpretation as the number of certain \emph{rc-graphs}
(also called \emph{pipe dreams}), see e.g.~\cite{BB}.  In particular,
we have $\Ups_w\in \nn$, even though this is not immediately apparent
from~\eqref{eq:macdonald} (cf.~$\S$\ref{ss:finrem-rc}).


Denote by $u(n)$ the largest principal specialization of the Schubert polynomial:
$$
u(n) \, := \, \max_{w \in S_n} \. \Ups_w\,.
$$

\begin{conjecture}[Stanley \cite{St}] \label{conj:stan}
There is a limit
$$\lim_{n\to \infty} \,\. \frac{1}{n^2} \, \log u(n)\ts.
$$
\end{conjecture}

In addition, Stanley asked whether the permutations $w$ in $S_n$
achieving the maximum \ts $\Ups_w = u(n)$ \ts had a limiting description.
There was some evidence in favor of this (see below), but before we turn
to positive results let us put this conjecture into context.

One can think of $\Ups_w$ as a statistical sum of
\emph{weighted random sorting networks} of
the permutation~$w$.  From a combinatorial point of view, this
is a more natural notion, since e.g. $\Ups_{w_0}=1$, where
$w_0=(n,n-1,\ldots,1)$ is the permutation with maximal length
\ts $\ell(w_0)=\binom{n}{2}$. It is thus natural to expect
$u(n)$ to have nice asymptotic behavior.  In fact, Stanley
gave the first order of asymptotics for~$u(n)$:

\begin{theorem}[Stanley \cite{St}] \label{thm:bound_u(n)}
\begin{equation} \label{eq:bounds-stanley}
\frac{1}{4} \, \leq \, \liminf_{n\to\infty} \. \frac{\log_2 u(n)}{n^2} \, \leq \,
\limsup_{n\to\infty} \. \frac{\log_2 u(n)}{n^2} \, \leq \, \frac{1}{2}\..
\end{equation}
\end{theorem}

Stanley's proof is nonconstructive and based on the
\emph{Cauchy identity} for Schubert polynomials, see~\cite[Prop.~2.4.7]{Man}.
The first constructive lower bound was given by the authors in \cite[$\S$6]{MPP},
where the asymptotics of $\Ups_w$ was computed for several families of permutations.
Notably, for a permutation
$$w(b,n-b) \. := \. \bigl(b, \ts b-1, \ts \ldots\ts ,\ts 1, \, n, \ts n-1, \ts \ldots\ts ,
\ts b+1\bigr) \quad \text{where} \ \ b\ts =\ts \frac{n}3\.,
$$
we showed that
$$
\frac{1}{n^2} \ts \log_2 \Ups_{w(b,n-b)} \, \lra \, C \. \approx \.
0.25162 \quad \text{as} \ \ \, n\to \infty\ts.
$$
In fact, it is easy to see that the limit $C$ is the largest limit value
over all ratios $0< b/n<1$.  This also gives a small improvement on the
lower bound in Stanley's theorem.

\emph{Layered permutations} $w(b_k,\ldots,b_1)$ are defined as
\[
w(b_k,b_{k-1},\ldots,b_1) \. := \. \bigl(b_k,b_k-1,\ldots, 1,
b_k+b_{k-1}, b_k+b_{k-1}-1,\ldots, b_k+1,\. \ldots \. , n,\ldots , \ts n-b_1+1\bigr),
\]
for integers \ts $b_1+\ldots+b_{k-1}+b_k =n$.  They are also called
\emph{Richardson} and \emph{pop-stack sortable permutations} in a different
contexts, see e.g.~\cite[$\S$2.1.4]{Kit} and~\cite{MeS}.
Denote by $\CR_n$ the set of layered permutations $w\in S_n$.

\smallskip

\begin{theorem} \label{thm:main-limit}
Let
$$
v(n) \, := \, \max_{w \in \CR_n} \. \Ups_w\ts.
$$
Then there is a limit
$$\lim_{n\to \infty} \,\. \frac{1}{n^2} \, \log_2 v(n) \, = \, \frac{ \gamma}{\log 2} \, \approx \, 0.2932362762\ts,
$$
where \ts $\gamma \approx 0.2032558981$ \ts is a universal constant.  Moreover,
the maximum value $v(n)$ is achieved at a layered permutation
$$
w(\,\ldots,\ts b_2,\ts b_1), \quad \text{where} \quad \ b_i \. \sim\.
\al^{i-1}(1-\al)\ts n \quad \text{as} \ \ \. n\to \infty,
$$
for every fixed~$i$, and where \ts $\al \approx 0.4331818312$ \ts is a universal constant.
\end{theorem}

\smallskip

In other words, the \emph{runs}~$b_i$ form a geometric distribution
in the limit.  See Figure~\ref{fig:limit_shape_richardson} for examples
of the  permutation matrix of such~$w$.  A posteriori this is unsurprising,
since the weights of reduced words are heavily skewed in favor of having
many transpositions at the end.

\begin{figure}[hbt]
\begin{minipage}{.54\textwidth}
  \centering
\includegraphics[height=3.8cm]{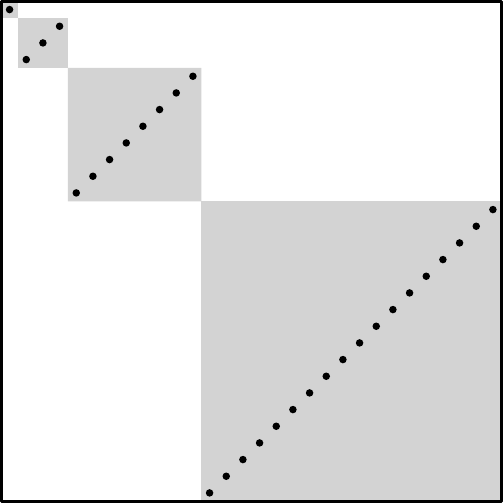}
\end{minipage}\!
\begin{minipage}{.49\textwidth}
  \centering
\includegraphics[height=3.8cm]{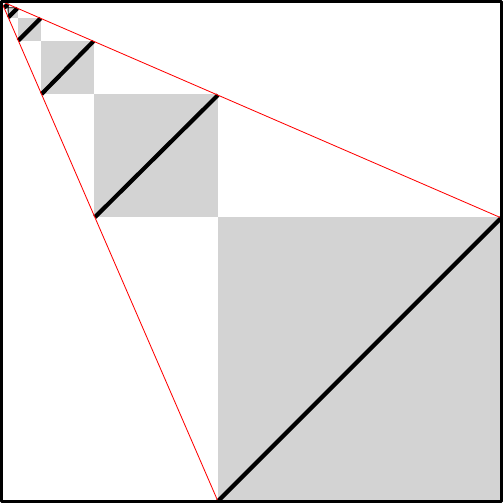}
\end{minipage}
\caption{Shapes of optimal layered permutations
 $w(1,3,8,18)$ and $w(2,4,9,20,46,106,246,567)$, of size $30$ and~$1000$, respectively.}
\label{fig:limit_shape_richardson}
\end{figure}

\smallskip

The story behind the theorem is also quite interesting.   Calculations for
$n\le 10$ reported in~\cite{MeS} and~\cite{St}, prompted Merzon and
Smirnov to make the following conjecture:\footnote{In 2026, Anderson, Panova, and Petrov \cite{APP} found two non-layered permutations of size $n=17$ that beat the optimal layered permutations and thus are counterexamples to the first part of Conjecture~\ref{conj:main}.}

\begin{conjecture}[{\cite[Conj.~5.7]{MeS}}] \label{conj:main}
For every~$n$, all permutations $w$ attaining the maximum $u(n)$ are
layered permutations. In particular, $u(n) = v(n)$.
\end{conjecture}

In other words, if the Merzon--Smirnov conjecture holds, our
Theorem proves Stanley's conjecture with the same limit value
and limiting description, as suggested by Stanley
(see~$\S$\ref{ss:finrem-exp} however).  Unconditionally,
Theorem~\ref{thm:main-limit}
improves a the lower bound for the \ts $\liminf$ \ts in
Theorem~\ref{thm:bound_u(n)} to about~$0.2932$.

\begin{remark}{\rm
We learned about the Merzon--Smirnov
conjecture from Hugh Thomas, who used it to compute $v(n)$
and permutations attaining it up to $n=300$ (see the Appendix).
This data allowed us to make a conjecture on the limit
shape, which we prove in the theorem.
}
\end{remark}

\medskip

\subsection*{Exact constants} The constants $\al$
and $\gamma$ in Theorem~\ref{thm:main-limit} are defined
as follows. Consider the function
\begin{equation} \label{eq:defr}
\fr(x) \, := \, x^2\log x \. - \.\frac12 \. (1-x)^2 \log(1-x)\.
- \. \frac12 \. (1+x)^2\log(1+x) \.  + \. 2\ts x\ts \log 2\ts.
\end{equation}
This function is obtained from a double integral that approximates the
logarithm of the product formula of Proctor \cite{Pr} for the number
of certain plane partitions (Proposition~\ref{p:bdF}). Then $\al$ is defined as the solution
other than $x=1$ of the equation
\[
2\ts x \ts \fr(x) \. + \. (1-x^2) \ts \fr'(x)\, = \, 0\ts,
\]
see Figure~\ref{fig:plots} for plots of $\fr(x)$ and the equation above.
The constant $\gamma$ is defined as
\[
\gamma \. := \. \frac{\fr(\al)}{1-\al^2}\..
\]
One can show that $\al$ is transcendental by using Baker's theorem,
see~\cite[$\S$2.1]{Baker}, but this goes beyond the scope of this paper.
It would be interesting to see if existing technology allows to show
that $\gamma$ is also transcendental.

\medskip

\subsection*{Outline of the paper}
 In Section~\ref{sec:background} we give the necessary
background on asymptotics and on the principal evaluation of
Schubert polynomials of layered permutations.
In Section~\ref{sec:proof} we prove Theorem~\ref{thm:main-limit}.
We conclude with final remarks and open problems
in Section~\ref{sec:finrem}.

\bigskip

\section{Background} \label{sec:background}

\subsection{Permutations}
We write permutations of $\{1,2,\ldots,n\}$
as $w=w_1w_2\ldots w_n \in S_n$, where $w_i$ is the
image of~$i$. Given two permutations $u$ in $S_m$ and $v$ in $S_n$ we
denote by $u\times v$ the following permutation of $S_{m+n}$:
\[
u \times v \. := \. u_1\ts u_2\ts \ldots \ts u_m \,  (m+v_1)\ts
(m+v_2) \. \ldots \. (m+v_n).
\]
Similarly, denote by \ts $1^m \times w$ \ts the permutation
$$
1^m \times w \. := \. 1\ts 2\ts \ldots \ts m \, (m+w_1)\ts (m+w_2)\.\ldots \. (m+w_n)\ts.
$$
Finally, let \ts$|b| = b_1+\cdots+b_k$.

\subsection{Product formulas for $\Ups_w$ for layered permutations}
In this section we give a product formula for
$\Ups_w$ when $w$ is a layered permutation $w(b_k,\ldots,b_1)$.

Let  $w_0$ be the longest permutation
$(p,p-1,\ldots,1)$ and let
$$
F(m,p)\, := \, \Ups_{1^m \times w_0}\ts.
$$
Fomin--Kirillov~\cite{FK} showed that $F(m,p)$ counts  the number of
plane partitions of shape $(p-1,p-2,\ldots,1)$ with entries at most~$m$.
This number of plane partitions has a product formula given by Proctor~\cite{Pr}.

\begin{theorem}[{\cite{FK,Pr}}] \label{thm:FK}
In the notation above, we have:
$$
F(m,p)
 \, = \, \prod_{1\leq i < j \leq p} \. \frac{2m+i+j-1}{i+j-1}\..
$$
\end{theorem}

\noindent
In notation of~\cite{MPP-asy}, we have:
$$
F(m,p)
 \, = \, \frac{\Lambda(2m+2p)\, \Lambda(2m+1) \,\Phi(p)
 }{\Phi(2m+p) \,\Lambda(2p)\, }\.,
$$
where $\Phi(n):=1!\cdot 2! \cdots (n-1)!$ and $\Lambda(n) := (n-2)!(n-4)!\cdots$

\begin{proposition}\label{prop:red}
For nonnegative integers $b_1,b_2,\ldots,b_k$, let $w(b_k,\ldots,b_1)$
be the associated layered permutation then
\begin{equation} \label{eq:keystep}
\Ups_{w(b_k,\ldots,b_1)}  \. = \. \Ups_{w(b_k,\ldots,b_2)} \cdot F\bigl(|b| -
b_1,b_1\bigr),
\end{equation}
where \ts $|b| = b_1+b_2+\cdots + b_k$.
\end{proposition}

\begin{proof}
The permutation $w(b_k,\ldots,b_1)$ can be written as the product
$w(b_k,\ldots,b_2) \times w_0(b_k)$. By properties of Schubert
polynomials (e.g. see \cite[(4.6)]{Mac} or \cite[Cor. 2.4.6]{Man}) we have that
\[
\mathfrak{S}_{w(b_k,\ldots,b_1)} = \mathfrak{S}_{w(b_k,\ldots,b_2)}
\cdot \mathfrak{S}_{1^{|b|-b_1} \times w_0(b_k)},
\]
and the result follows by doing a principal evaluation.
\end{proof}

\begin{remark} {\rm
Equation~\eqref{eq:keystep} can be turned into a dynamic program to find
layered permutations $w(b_k,\ldots,b_1)$ that achieves $v(n)$, see the
appendix.}
\end{remark}

\bigskip

\section{Asymptotics of the largest $v(n)$} \label{sec:proof}

\subsection{The outline}
We will use \eqref{eq:keystep} inductively to prove the main
result.  Let $p := b_1$ and
$m:=n-p$, so that \ts $m= b_2+\ldots+b_k$.
By definition of~$v(n)$, we have that
\[
v(n) \. = \, \max_{b \,:\, |b|=n} \ts \Ups_{w(b)}\ts.
\]
Next, using \eqref{eq:keystep}, $v(n)$ becomes
\begin{equation} \label{eq:defvn}
v(n) \. = \. \max_{1\le p \le n} \ts \bigl\{ v(n-p) \ts F(n-p,p) \bigr\}.
\end{equation}

We will need very precise estimates on $\log F(m,n-m)$.
Note that the exact asymptotic expansion for the
\emph{Barnes $G$-function}, which can be used to obtain
the asymptotics of $\Phi(\cdot)$ and $\La(\cdot)$,
see e.g.~\cite{AR}.  However, these bounds are insufficient as
we also need sharp bounds for the error terms which hold for
all~$m$ and $n$.  We obtain these in the
next subsection. These estimates are then combined with
Proposition~\ref{prop:red} to prove Theorem~\ref{thm:main-limit}.

\medskip

\subsection{Technical estimates}
Let $\fr(x)$ be the function  defined in
\eqref{eq:defr}. The next lemma gives bounds on \ts $\log
F(m,n-m)$ \ts in terms of the function~$\fr(x)$.

\begin{proposition} \label{p:bdF} For all integers \ts $n \ge m \ge 0$, we have:
\[
-2 \ts n \. \leq \, \log F(m,n-m) \ts - \ts n^2 \fr(m/n) \. \leq \. 0.
\]
\end{proposition}

We split the proof into two lemmas, one for the
upper bound and the other for the lower bound.

\begin{lemma} \label{l:UB}
 For all integers \ts $n \ge m \ge 0$, we have:
\[
\log F(m,n-m) - n^2 \fr(m/n) \leq 0.
\]
\end{lemma}

\begin{proof}
We use the product formula
for $F(m,p)$ in Theorem~\ref{thm:FK}.
\begin{align}
\log F(m,p) &= \sum_{1\leq i<j\leq p} \Bigl( \log(2m+i+j-1) -
  \log(i+j-1) \Bigr) \notag \\
&= \sum_{1\leq i\leq j'\leq p-1} \Bigl( \log(2m+i+j') -
  \log(i+j') \Bigr), \label{eq:sumbound}
\end{align}
where we changed the index to $j'=j-1$. Next, we approximate this sum
using a double integral. Let
\[
g(x,y):= \log(2m+x+y)-\log(x+y).
\]
Notice that the function $g(x,y)$ is
constant along the lines $x+y=k$ for constant $k$. Therefore, we can
shift the terms of the sum in the RHS of \eqref{eq:sumbound} by $(i,j)\mapsto
(i-1/\sqrt{2},j+1/\sqrt{2})$ without changing the sum (see center of
Figure~\ref{fig:triangle})
\begin{equation} \label{eq:sumbound2}
\log F(m,p)  \, = \, \sum_{(i,j) \in S} \ts \Bigl( \log(2m+i+j') -
  \log(i+j') \Bigr),
\end{equation}
where \ts $S =\bigl\{\mathbb{Z}^2 + (-1/\sqrt{2},1/\sqrt{2})\bigr\} \ts \cap \ts \bigl\{ (x,y)~:~0\leq x \leq p,
x <  y \leq p\bigr\}$.

\begin{figure}[hbt]
\begin{center}
\includegraphics[scale=0.7]{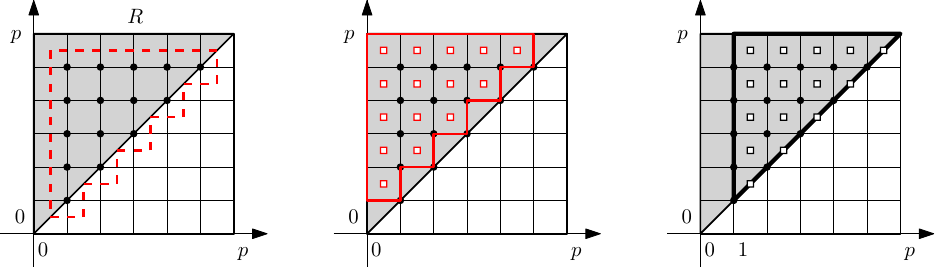}
\caption{Illustration of the proof of the upper and lower bounds of
  Proposition~\ref{p:bdF} for $\log F(m,p)$ for $p=5$. The lattice
  points $\bullet$ on the left are the support of the sum $\sum_{i
  \leq j} \. g(i,j)$. This sum remains the same if the support is
shifted by $\bigl(\frac{-1}{\sqrt{2}},\frac{1}{\sqrt{2}}\bigr)$, giving
points ${\footnotesize \textcolor{red}{\Box}}$ in the
middle. The original sum is bounded below by the sum over the support shifted by
$\bigl(\frac{1}{\sqrt{2}},\frac{1}{\sqrt{2}}\bigr)$, giving points ${\footnotesize \Box}$ in the right.}
\label{fig:triangle}
\end{center}
\end{figure}

Next, compute the Hessian \ts $\rH$ \ts of $g(x,y)$. We have:
\[
\rH \, = \, C \cdot \begin{bmatrix} 1 & 1 \\ 1 & 1 \end{bmatrix}, \quad \text{where} \quad
C \. = \. \frac{1}{(x+y)^2} \. - \.\frac{1}{(2\ts m+x+y)^2} \,\..
\]
Matrix $\rH$ has eigenvalues \ts $0$ \ts and \ts $2\ts C$ \ts that are nonnegative in
\ts $[0,p] \times [0,p]$.
Thus $g(x,y)$ is convex in this region. The modified sum in
\eqref{eq:sumbound2} is the sum of values of $g(x,y)$ over centers of
the unit squares which fit entirely in $R$. By convexity, each such value
of $g(x,y)$ is less than the average value of $g(x,y)$ over its
square. Hence the sum in
\eqref{eq:sumbound2} is bounded above by the double integral,
\begin{align*}
\log F(m,p) \leq \int_{0}^p\int_{y}^p \left(\log(2m+x+y)-\log(x+y)\right)\. dx\ts dy.
\end{align*}
Next, we compute this double integral and obtain
\begin{equation} \label{eq:origin_f}
\int_{0}^p\int_{y}^p \ts \left(\log(2m+x+y)-\log(x+y)\right)\. dx\ts dy \. = \.
(m+p)^2 \fr(m/(m+p)),
\end{equation}
for $\fr(x)$ defined in \eqref{eq:defr}. This proves the upper bound.
\end{proof}

\begin{lemma} \label{l:LB}
 For all integers \ts $n \ge m \ge 0$, we have:
\[
\log F(m,n-m) - n^2 \fr(m/n) \. \geq \. -2\ts n.
\]
\end{lemma}

\begin{proof}
Since the function $g(x,y)$ is decreasing along
the $x$ direction and $y$ direction then each value $g(i,j)$ in the
sum is bigger than the average value of $g(x,y)$ over the unit square
with center $(i+1/\sqrt{2},j+1/\sqrt{2})$ (see right of
Figure~\ref{fig:triangle}). Hence the original sum in
\eqref{eq:sumbound} is bounded below by the double integral

\begin{equation} \label{eq:keylowerbound}
\log F(m,p) \, = \, \sum_{1\leq i<j\leq p} g(i,j) \geq \int_1^p\int_x^p g(x,y) \. dy\ts dx.
\end{equation}
This integral can be written in terms of the original integral,
computed in \eqref{eq:origin_f}, as follows
\begin{align}
 \int_1^p\int_x^p g(x,y) \. dy\ts dx \, &  = \, \int_0^p\int_x^p g(x,y) \. dy\ts dx
\, - \, \int_0^1\int_x^p g(x,y) \. dy\ts dx \notag\\
&= \, (m+p)^2 \fr(m/(m+p)) \, -  \, \int_0^1\int_x^p g(x,y) \. dy\ts dx. \label{eq:pflowerbound_1}
\end{align}
Since the function $g(x,y)$ is decreasing in the $x$ direction then
the double integral in the RHS above is bounded by the following single integral
\begin{equation}
-\int_0^1\int_x^p g(x,y) \. dy\ts dx \, \geq \, -\int_0^p g(0,y) \. dy.
\end{equation}  \label{eq:pflowerbound_2}
We evaluate this single integral and use Jensen's inequality to obtain
\begin{align}
- \int_0^p g(0,y) dy \, & = \,
2m\log(2m) \. + \. p\log(p) \. - \. (2m+p)\log(2m+p) \notag \\
&\geq \, (2m+p)\bigl( \log(2m+p) \. - \. \log 2\bigr) \. -
  \. (2m+p)\log(2m+p). \label{eq:pflowerbound_3}
\end{align}
Combining \eqref{eq:keylowerbound},\eqref{eq:pflowerbound_1}, \eqref{eq:pflowerbound_2}, and \eqref{eq:pflowerbound_3} we have
\[
\log F(m,p) \. \geq \. (m+p)^2 \fr(m/(m+p)) \. + \. (2m+p)\left( \log(2m+p) \.
- \.  \log(2)\right) \. - \. (2m+p)\log(2m+p).
\]
The RHS is greater than or equal to \ts $(m+p)^2 \fr(m/(m+p)) -2(m+p)$,
as desired.
\end{proof}

\medskip

\subsection{Optimizing constants}
Our goal is to show that $\lim_{n\to
  \infty} \log_2 v(n)/n^2$ is a constant. In the previous lemma we gave bounds on the error of approximating
$\log F(m,n-m)$ by $n^2\fr(x)$ where $x = m/n$ in $[0,1]$. We now find a unique constant
$\gamma$ such that $\fr(x)+\gamma x^2$ has a unique maximum over $x \in [0,1)$.

\begin{lemma}\label{lemma:max_r}
There exist a unique $\gamma >0$ and $\al \in (0,1)$, such that:
\begin{itemize}
\item[(1)] $2\gamma \al + \fr'(\al) \ts =\ts 0$,
\item[(2)] $\gamma \al^2 + \fr(\al) \ts = \ts \gamma$ \ts with \ts $2\gamma + r''(\al) \ts \leq \ts 0$.
\end{itemize}
And for this $\gamma$, the maximum of \ts $\fr(x) + \gamma x^2$ \ts over $x\in [0,1)$ is achieved at
the given $\al$, and the value is precisely $\gamma$. That is,
\[
\max_{x\in [0,1)}(\fr(x) + \gamma x^2) \. = \. \fr(\al) + \gamma \al^2 \. = \. \gamma.
\]
\end{lemma}

\begin{proof}
First, it is straightforward to show that $\lim_{x\to 0} \fr(x)  =
\lim_{x\to 1} \fr(x) =0$ and that $\fr(x)>0$ for $x\in(0,1)$ (see plot of
$\fr(x)$ on the left of Figure~\ref{fig:plots}).


Let $\al$ be a solution to the equation $q(x)=0$ where
\begin{align*}
q(x) &:=\,\, \fr(x) 2x + \fr'(x)(1-x^2)\\
  &= (1-x)^2\log(1-x) - (1+x)^2\log(1+x) + 2x\log(x) + 2(1+x^2)\log(2).
\end{align*}
This function on the RHS above has one root $\al=0.4331818312..$ and the other is $x=1$, as easily seen from the plot, but also can be shown analytically.
Then we set $$\gamma := \frac{\fr(\al)}{1-\al^2} = -\frac{\fr'(\al)}{2\al},$$ so $\gamma$ and $\al$ now satisfy conditions (1) and (2).

\begin{figure}
\includegraphics[width=1.6in]{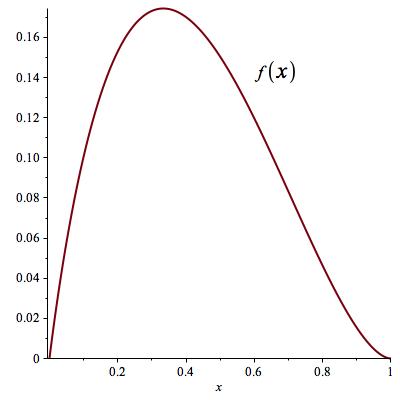}\quad \includegraphics[width=1.6in]{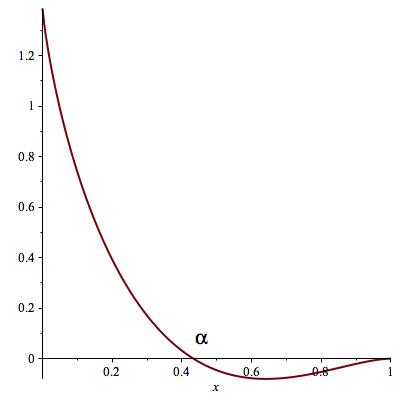}\quad \includegraphics[width=1.6in]{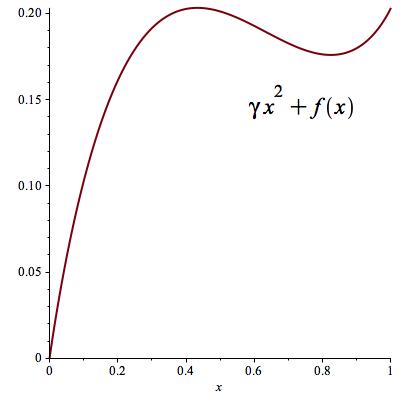}
\caption{Graphs of the functions $\fr(x)$, $q(x)$ and $\gamma x^2+\fr(x)$, on $(0,1)$.}
\label{fig:plots}
\end{figure}

Next, we see that $\gamma = \fr(\al)/(1-\al^2) \approx
0.2032558981$. To prove that this is indeed a maximum for $\fr(x)+\gamma
x^2$, we check that the second derivative,  $d^2(\gamma x^2 +
\fr(x))/dx^2 =2\gamma + r''(x)<0$ for $x=\al$. We have that $r''(x) = \log
(x^2/(1-x^2))$. Since $\al \leq 0.45$, we have that $
x^2/(1-x^2) <0.26$ and so $r''(\al) <-1.3<-2\gamma $ and so the
value is a local maximum and by condition~(2) it is equal to $\gamma$.
\end{proof}

\medskip

\subsection{Proof of Theorem~\ref{thm:main-limit}}
The theorem follows immediately from the following lemma.

\begin{lemma}\label{l:limit}
For all $n\ge 2$ we have:
\[
\bigl| \log v(n) - \gamma n^2\bigr| \. \leq \. 4 n\ts.
\]
Conversely, suppose for a layered permutation $w(b)\in S_n$ we have
\[
\bigl| \log \Ups_w - \gamma n^2\bigr| \. \leq \. 4 n\ts.
\]
Then $b=(....,b_2,b_1)$, s.t. $b_i \sim (1-\al)\al^{i-1} n$ for all
fixed $i\ge 1$.
\end{lemma}

\begin{proof}
We proceed by induction to show that \ts $|\log v(n) - \gamma n^2| \leq 4 \ts n$
\ts holds for all $n\ge 2$.  The base cases $n=2$ can be checked directly
(see exact values in the appendix).

We start with \eqref{eq:defvn} and use the induction hypothesis and the upper bound of
Proposition~\ref{p:bdF} to obtain
\begin{align*}
\log v(n) \. &= \. \max_{m < n} \ts \bigl(\log v(m) + \log F(m,n-m) \bigr)\\
& \leq \. \max_{m<n} \ts \bigl( \gamma m^2 + \log F(m,n-m) +2m \bigr),\\
& \leq  n^2\max_{x\in[0,1)} \bigl(\fr(x) + \gamma x^2\bigr) \. + \. 2 \ts n.
\end{align*}
By Lemma~\ref{lemma:max_r}, the maximum value of
$\fr(x)+\gamma x^2$ is equal to~$\gamma$. Thus, the above inequality becomes
\[
\log v(n) \. \leq \. \gamma n^2 \ts +\ts 2\ts n.
\]
This maximum is achieved when $x=\al$, i.e.\ when $m=n\al$ and $p=b_1 = (1-\al)n$. By the
definition of~$v(n)$, for this value of~$m$ we have that
\[
\log v(n) \. \geq \. \log v(n\al) \ts +\ts \log F(n\al,n-n\al).
\]
By the induction hypothesis and the lower bound of
Proposition~\ref{p:bdF}, the above inequality becomes
\begin{align*}
\log v(n) \. & \geq \. \bigl(\gamma n^2\al^2 \ts - \ts 4 \ts n\al \bigr) \ts +\ts (n^2 \fr(\al) \ts -\ts 2\ts n) \\
 &= \. \gamma n^2 \ts  - \ts 2\ts (1\ts+\ts 2\ts \al)\ts n\ts \geq \gamma n^2 \ts - \ts 4\ts n.
\end{align*}
Here we again used the fact that \ts $\fr(\al) + \gamma \al^2 = \gamma$ and that $\al \leq 1/2$.
In summary,
$$\bigl|\log v(n) - \gamma n^2 \bigr| \. \leq \. 4 \ts n ,
$$
and this bound is attained when \ts $b_1 \sim (1-\al)n$. Recursively,
we obtain \ts $b_i \sim (1-\al)\al^{i-1}n$ for every fixed $i=2,3,\ldots$.
\end{proof}

\begin{remark}{\rm
Note that the appendix shows rather slow rate of convergence
for \ts $h(n) := \frac{1}{n^2} \ts \log_2 v(n)$, giving only
$h(300)\approx 0.2904$.  This suggests that \ts
$h(n) = \ga/(\log 2) - 1/n - o(1/n)$, so that the bound in
Lemma~\ref{l:limit} is quite sharp.
}
\end{remark}

\bigskip

\section{Final remarks}\label{sec:finrem}

\subsection{} \label{ss:finrem-limit}
Stanley's Conjecture~\ref{conj:stan} remains open but is very
likely to hold. Denote by
$$
a(n) \. = \. \sum_{w \in S_n} \. \Ups_w
$$
the total number of rc-graphs (pipe dreams) of size~$n$.  Since
$$
u(n) \. \le \. a(n) \. \le \. n! \ts u(n),
$$
we conclude that it suffices to prove the asymptotics result
for~$a_n$. This suggests connections to counting general tilings
(see e.g.~\cite{AS}), as pipe dreams can be viewed as tilings of
a staircase shape with two types of tiles, but with one global
condition (strains can intersect at most once).  The problem is
especially similar to counting \emph{Knutson-Tao puzzles}
enumerating the \emph{Littlewood--Richardson coefficients},
whose maximal asymptotics was recently studied in~\cite{PPY}.


By analogy with the tilings, one can ask if $u(n)$ satisfies some sort of
super-multiplicativity property.  Formally, let $w\otimes 1^c$ denote
the \emph{Kronecker product permutation} of size $\ts c\ts n\ts$,
whose permutation matrix equals the Kronecker
product of the permutation matrix $P_w$ and the identity~$I_c$
(see~\cite{MPP}).

\begin{conjecture}
For $w\in S_n$, we have \ts $\Ups_{w \otimes 1^2} \ts \geq \ts \Ups_w^4$.
\end{conjecture}

We verified the conjecture for all $w\in S_n$ where $n\le 5$, but perhaps
more computational evidence would be helpful.

\smallskip

\subsection{} \label{ss:finrem-exp}
Similarly, the Merzon--Smirnov Conjecture~\ref{conj:main} remains open.
In our opinion, the numerical evidence in favor of the conjecture is
insufficient, and it would be interesting to verify it for larger~$n$.
To speedup the computation, perhaps, there are large classes of
permutations $u\in S_n$ which can be proved to be non-maximal, i.e.\
there exists \ts $w\in S_n$, s.t. $\Ups_u\leq \Ups_{w}$.
Such permutations can then be ignored in the exhaustive search.

In fact, Prop.~6.5 in~\cite{MPP} gives explicit
constructions of large families of permutations $w\in S_n$,
for which  $\log \Ups_w = \Theta(n)$.  These permutations are
very far from being layered (in the transposition distance),
suggesting that if true, proving Conjecture~\ref{conj:main}
might not be easy.

\smallskip

\subsection{}
In \cite{St}, Stanley also considered the case when $\Ups_w$ is
small. It is well known that $\Ups_w=1$ if and only if $w$ is
{\em dominant}~\cite{Man}, i.e.\ $132$-avoiding. Stanley conjectured
that $\Ups_w=2$ if and only if $w$ has exactly one instance of the
pattern~$132$. This was recently proved by Weigandt~\cite{W}, who also
showed that $\Ups_w-1$ is greater than or equal the number of
instances of the pattern $132$ in~$w$.

This suggests the problem of finding permutations where
the number of patterns $132$ is maximal.  In the field of
pattern avoidance, this problem can be rephrased as asking for
permutations $w\in S_n$ with maximal \emph{packing density}
of the pattern $132$, see~\cite[$\S$8.3.1]{Kit}.  The
solution due to Stromquist is extremely well understood,
and has been both refined and generalized,
see~\cite{A+,BSV,HSV}, \cite[$\S$5.1]{Price} and
\cite[\href{https://oeis.org/A061061}{A061061}]{OEIS}.
The maximal packing density is attained at a layered
permutation \ts $w(b_1,b_2,\ldots)$, where the runs $b_i$ have a
geometric distribution:
$$
b_i\. \sim \. \rho\ts (1-\rho)^{i-1} \ts n, \ i=1,2,\ldots \quad \text{where} \quad
\rho \. =  \. \frac{\sqrt{3}-1}{2} \. \approx \.  0.366025
$$

While, of course, $v(n)$ are attained at somewhat different layered
permutations, the similarities to this work are rather striking
and go beyond coincidences.  They are rooted in the recursive nature
of optimal permutations in both cases, which are solutions of
similar (but different!) maximization problems.


\smallskip

\subsection{}\label{ss:finrem-rc}
The bounds for $u(n)$ from Theorem~\ref{thm:bound_u(n)} are obtained
from the Cauchy identity of Schubert polynomials which gives
\begin{equation} \label{eq:CauchySchubert}
\sum_{w_0 = v^{-1}u}\Ups_u \Ups_v = 2^{\binom{n}{2}}.
\end{equation}
One could then ask for large values of $\Ups_w\Ups_{w w_0^{-1}}$. Let $u'(n) := \max_{w\in S_n} \{ \Ups_w\cdot\Ups_{w w_0^{-1}}
\}$. The table below has the values of $u'(n)$ for $n=2,\ldots,9$ and
the permutations $w$ (up to multiplying by $w_0^{-1}$) that achieve
that value $u'(n)$.
\[
\begin{array}{rrc}
\hline
n & u'(n) & w \\ \hline
3 & 2 & 132\\
4 & 6 & 1423\\
5 & 33 & 15243\\
6 & 286 & 162534\\
7 & 4620 & 1736254\\
8 & 162360 &  18527364 \\
9 & 9057090  & 195283746 \\ \hline
\end{array}
\]
Note that for a layered permutation $w(b)$, the permutation $w(b)w_0^{-1}$ is dominant and so $\Ups_{w(b) w_0^{-1}}=1$.

There is a combinatorial proof of \eqref{eq:CauchySchubert} by
Bergeron and Billey \cite{BB} involving taking a double {\em $rc$-graph} of
$w_0$ ($2^{\binom{n}{2}}$ many) and reading from each half of it permutations $u$
and $v$ satisfying $w_0 = v^{-1}u$. All such double $rc$-graphs of $w_0$ can be
obtained from an initial double $rc$-graph via
certain local transformations (see \cite[Sec.~4]{BB}). One can use these local
transformations in a Markov chain to obtain a random double $rc$-graph
of $w_0$ and from it read off a permutation~$u$; see
Figure~\ref{fig:random-permdrc}. We conjecture that
the permutation matrix of random permutations~$u$ has
a parabolic frozen region.

The second permutation in Figure~\ref{fig:random-permdrc} is obtained
by running a Markov chain for $5\cdot 10^9$ local moves
on a double $rc$-graph of $v^{-1}u=w_0 \in S_{50}$,
described in~\cite[Sec.~4]{BB}.  Half of the resulting
double $rc$-graph given in Figure~\ref{fig:random-permdrc-app}
is then converted into a permutation $u\in S_{50}$.

\begin{figure}
\begin{center}
\includegraphics{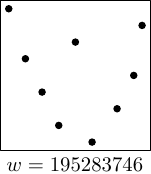}\qquad
\qquad\includegraphics[scale=0.32]{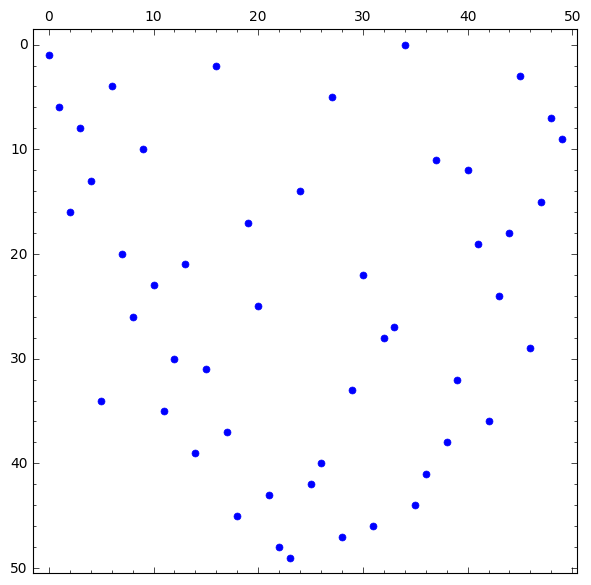}
\caption{Permutation matrices  of $195283746$ and of a permutation
$u\in S_{50}$ from the random double $rc$-graph.}
\label{fig:random-permdrc}
\end{center}
\end{figure}

\vskip.7cm

\subsection*{Acknowledgements}
We are grateful to Nantel Bergeron, Sara Billey, Grigory Merzon, Evgeny Smirnov, Richard Stanley
and Damir Yeliussizov for helpful conversations.  We are especially
thankful to Hugh Thomas for sharing the data of layered permutations.
The calculations of $\Ups_w$ done in this paper were made using Sage and
its algebraic combinatorics features developed by the Sage-Combinat
community~\cite{sage-combinat}.
The second and third authors were partially supported by the~NSF.


\newpage

\appendix
{\normalsize
\subsection*{\large Table of exact values for $n\le 300$}
Below we present table of tuples $b$ of layered permutations $w(b)$
  maximizing $v(n)$.   The third column is \ts
  $f(n):= \frac{1}{n^2} \ts \log_2 v(n)$. }

\bigskip

{\small
\begin{minipage}{0.5\textwidth}
$$
\begin{array}{rrr}
\hline
n & \quad \ (\ldots,b_2,b_1) & \qquad \quad f(n)\\ \hline
1&(1)&0.000000\\
2&(1, 1)&0.000000\\
3&(1, 2)&0.111111\\
4&(1, 3)&0.145121\\
5&(1, 1, 3)&0.152294\\
6&(1, 1, 4)&0.177564\\
7&(1, 2, 4)&0.191149\\
8&(1, 2, 5)&0.206317\\
9&(1, 2, 6)&0.213824\\
10&(1, 3, 6)&0.220771\\
11&(1, 3, 7)&0.227005\\
12&(1, 3, 8)&0.229879\\
13&(1, 1, 3, 8)&0.233769\\
14&(1, 1, 4, 8)&0.237048\\
15&(1, 1, 4, 9)&0.241677\\
16&(1, 1, 4, 10)&0.244446\\
17&(1, 2, 4, 10)&0.246954\\
18&(1, 2, 4, 11)&0.249509\\
19&(1, 2, 5, 11)&0.251966\\
20&(1, 2, 5, 12)&0.254240\\
21&(1, 2, 5, 13)&0.255575\\
22&(1, 2, 6, 13)&0.257354\\
23&(1, 2, 6, 14)&0.258685\\
24&(1, 3, 6, 14)&0.260063\\
25&(1, 3, 6, 15)&0.261360\\
26&(1, 3, 7, 15)&0.262425\\
27&(1, 3, 7, 16)&0.263673\\
28&(1, 3, 7, 17)&0.264435\\
29&(1, 3, 8, 17)&0.265233\\
30&(1, 3, 8, 18)&0.266034\\
31&(1, 1, 3, 8, 18)&0.266811\\
32&(1, 1, 3, 8, 19)&0.267619\\
33&(1, 1, 4, 8, 19)&0.268165\\
34&(1, 1, 4, 8, 20)&0.268973\\
35&(1, 1, 4, 9, 20)&0.269675\\
36&(1, 1, 4, 9, 21)&0.270460\\
37&(1, 1, 4, 9, 22)&0.270978\\
38&(1, 1, 4, 10, 22)&0.271548\\
39&(1, 1, 4, 10, 23)&0.272081\\
40&(1, 2, 4, 10, 23)&0.272523\\
41&(1, 2, 4, 10, 24)&0.273065\\
42&(1, 2, 4, 11, 24)&0.273453\\
43&(1, 2, 4, 11, 25)&0.273996\\
44&(1, 2, 4, 11, 26)&0.274357\\
45&(1, 2, 5, 11, 26)&0.274862\\
46&(1, 2, 5, 11, 27)&0.275235\\
47&(1, 2, 5, 12, 27)&0.275654\\
48&(1, 2, 5, 12, 28)&0.276036\\
49&(1, 2, 5, 12, 29)&0.276277\\
50&(1, 2, 5, 13, 29)&0.276634\\ \hline
\end{array}
$$
\end{minipage}
\begin{minipage}{0.25\textwidth}
$$
\begin{array}{rrr}
\hline
n & \qquad \ \ (\ldots,b_2,b_1) & \qquad \quad f(n)\\ \hline
51&(1, 2, 5, 13, 30)&0.276896\\
52&(1, 2, 6, 13, 30)&0.277275\\
53&(1, 2, 6, 13, 31)&0.277550\\
54&(1, 2, 6, 14, 31)&0.277807\\
55&(1, 2, 6, 14, 32)&0.278094\\
56&(1, 3, 6, 14, 32)&0.278322\\
57&(1, 3, 6, 14, 33)&0.278618\\
58&(1, 3, 6, 14, 34)&0.278815\\
59&(1, 3, 6, 15, 34)&0.279103\\
60&(1, 3, 6, 15, 35)&0.279313\\
61&(1, 3, 7, 15, 35)&0.279525\\
62&(1, 3, 7, 15, 36)&0.279747\\
63&(1, 3, 7, 16, 36)&0.279962\\
64&(1, 3, 7, 16, 37)&0.280192\\
65&(1, 3, 7, 16, 38)&0.280344\\
66&(1, 3, 7, 17, 38)&0.280532\\
67&(1, 3, 7, 17, 39)&0.280698\\
68&(1, 3, 8, 17, 39)&0.280862\\
69&(1, 3, 8, 17, 40)&0.281038\\
70&(1, 3, 8, 18, 40)&0.281178\\
71&(1, 3, 8, 18, 41)&0.281363\\
72&(1, 3, 8, 18, 42)&0.281486\\
73&(1, 1, 3, 8, 18, 42)&0.281670\\
74&(1, 1, 3, 8, 18, 43)&0.281803\\
75&(1, 1, 3, 8, 19, 43)&0.281969\\
76&(1, 1, 3, 8, 19, 44)&0.282112\\
77&(1, 1, 4, 8, 19, 44)&0.282210\\
78&(1, 1, 4, 8, 19, 45)&0.282361\\
79&(1, 1, 4, 8, 20, 45)&0.282488\\
80&(1, 1, 4, 8, 20, 46)&0.282646\\
81&(1, 1, 4, 8, 20, 47)&0.282755\\
82&(1, 1, 4, 9, 20, 47)&0.282902\\
83&(1, 1, 4, 9, 20, 48)&0.283019\\
84&(1, 1, 4, 9, 21, 48)&0.283165\\
85&(1, 1, 4, 9, 21, 49)&0.283288\\
86&(1, 1, 4, 9, 22, 49)&0.283370\\
87&(1, 1, 4, 9, 22, 50)&0.283501\\
88&(1, 1, 4, 9, 22, 51)&0.283590\\
89&(1, 1, 4, 10, 22, 51)&0.283715\\
90&(1, 1, 4, 10, 22, 52)&0.283811\\
91&(1, 1, 4, 10, 23, 52)&0.283914\\
92&(1, 1, 4, 10, 23, 53)&0.284018\\
93&(1, 2, 4, 10, 23, 53)&0.284090\\
94&(1, 2, 4, 10, 23, 54)&0.284200\\
95&(1, 2, 4, 10, 24, 54)&0.284279\\
96&(1, 2, 4, 10, 24, 55)&0.284394\\
97&(1, 2, 4, 10, 24, 56)&0.284475\\
98&(1, 2, 4, 11, 24, 56)&0.284553\\
99&(1, 2, 4, 11, 24, 57)&0.284641\\
100&(1, 2, 4, 11, 25, 57)&0.284736\\
\hline
\end{array}
$$
\end{minipage}

\newpage

\begin{minipage}{0.5\textwidth}
$$
\begin{array}{rrr}
\hline
n & \qquad \quad  (\ldots,b_2,b_1) & \qquad \quad f(n)\\ \hline
101&(1, 2, 4, 11, 25, 58)&0.284828\\
102&(1, 2, 4, 11, 25, 59)&0.284891\\
103&(1, 2, 4, 11, 26, 59)&0.284978\\
104&(1, 2, 4, 11, 26, 60)&0.285046\\
105&(1, 2, 5, 11, 26, 60)&0.285148\\
106&(1, 2, 5, 11, 26, 61)&0.285222\\
107&(1, 2, 5, 11, 27, 61)&0.285289\\
108&(1, 2, 5, 11, 27, 62)&0.285368\\
109&(1, 2, 5, 12, 27, 62)&0.285434\\
110&(1, 2, 5, 12, 27, 63)&0.285518\\
111&(1, 2, 5, 12, 27, 64)&0.285577\\
112&(1, 2, 5, 12, 28, 64)&0.285657\\
113&(1, 2, 5, 12, 28, 65)&0.285720\\
114&(1, 2, 5, 12, 29, 65)&0.285766\\
115&(1, 2, 5, 12, 29, 66)&0.285834\\
116&(1, 2, 5, 13, 29, 66)&0.285892\\
117&(1, 2, 5, 13, 29, 67)&0.285965\\
118&(1, 2, 5, 13, 29, 68)&0.286015\\
119&(1, 2, 5, 13, 30, 68)&0.286074\\
120&(1, 2, 5, 13, 30, 69)&0.286129\\
121&(1, 2, 6, 13, 30, 69)&0.286201\\
122&(1, 2, 6, 13, 30, 70)&0.286261\\
123&(1, 2, 6, 13, 31, 70)&0.286306\\
124&(1, 2, 6, 13, 31, 71)&0.286369\\
125&(1, 2, 6, 13, 31, 72)&0.286413\\
126&(1, 2, 6, 14, 31, 72)&0.286472\\
127&(1, 2, 6, 14, 31, 73)&0.286519\\
128&(1, 2, 6, 14, 32, 73)&0.286576\\
129&(1, 2, 6, 14, 32, 74)&0.286628\\
130&(1, 3, 6, 14, 32, 74)&0.286667\\
131&(1, 3, 6, 14, 32, 75)&0.286723\\
132&(1, 3, 6, 14, 33, 75)&0.286768\\
133&(1, 3, 6, 14, 33, 76)&0.286827\\
134&(1, 3, 6, 14, 33, 77)&0.286868\\
135&(1, 3, 6, 14, 34, 77)&0.286910\\
136&(1, 3, 6, 14, 34, 78)&0.286955\\
137&(1, 3, 6, 15, 34, 78)&0.287007\\
138&(1, 3, 6, 15, 34, 79)&0.287056\\
139&(1, 3, 6, 15, 34, 80)&0.287089\\
140&(1, 3, 6, 15, 35, 80)&0.287140\\
141&(1, 3, 6, 15, 35, 81)&0.287177\\
142&(1, 3, 7, 15, 35, 81)&0.287221\\
143&(1, 3, 7, 15, 35, 82)&0.287262\\
144&(1, 3, 7, 15, 36, 82)&0.287303\\
145&(1, 3, 7, 15, 36, 83)&0.287346\\
146&(1, 3, 7, 16, 36, 83)&0.287380\\
147&(1, 3, 7, 16, 36, 84)&0.287427\\
148&(1, 3, 7, 16, 36, 85)&0.287459\\
149&(1, 3, 7, 16, 37, 85)&0.287508\\
150&(1, 3, 7, 16, 37, 86)&0.287544\\ \hline
\end{array}
$$
\end{minipage}
\begin{minipage}{0.5\textwidth}
$$
\begin{array}{rrr}
\hline
n & \qquad\qquad  (\ldots,b_2,b_1) & \qquad \quad f(n)\\ \hline
151&(1, 3, 7, 16, 38, 86)&0.287573\\
152&(1, 3, 7, 16, 38, 87)&0.287612\\
153&(1, 3, 7, 17, 38, 87)&0.287643\\
154&(1, 3, 7, 17, 38, 88)&0.287684\\
155&(1, 3, 7, 17, 38, 89)&0.287713\\
156&(1, 3, 7, 17, 39, 89)&0.287750\\
157&(1, 3, 7, 17, 39, 90)&0.287782\\
158&(1, 3, 8, 17, 39, 90)&0.287814\\
159&(1, 3, 8, 17, 39, 91)&0.287849\\
160&(1, 3, 8, 17, 40, 91)&0.287879\\
161&(1, 3, 8, 17, 40, 92)&0.287916\\
162&(1, 3, 8, 17, 40, 93)&0.287942\\
163&(1, 3, 8, 18, 40, 93)&0.287975\\
164&(1, 3, 8, 18, 40, 94)&0.288003\\
165&(1, 3, 8, 18, 41, 94)&0.288040\\
166&(1, 3, 8, 18, 41, 95)&0.288071\\
167&(1, 3, 8, 18, 42, 95)&0.288093\\
168&(1, 3, 8, 18, 42, 96)&0.288126\\
169&(1, 1, 3, 8, 18, 42, 96)&0.288155\\
170&(1, 1, 3, 8, 18, 42, 97)&0.288191\\
171&(1, 1, 3, 8, 18, 42, 98)&0.288216\\
172&(1, 1, 3, 8, 18, 43, 98)&0.288244\\
173&(1, 1, 3, 8, 18, 43, 99)&0.288272\\
174&(1, 1, 3, 8, 19, 43, 99)&0.288302\\
175&(1, 1, 3, 8, 19, 43, 100)&0.288332\\
176&(1, 1, 3, 8, 19, 44, 100)&0.288355\\
177&(1, 1, 3, 8, 19, 44, 101)&0.288387\\
178&(1, 1, 3, 8, 19, 44, 102)&0.288410\\
179&(1, 1, 4, 8, 19, 44, 102)&0.288432\\
180&(1, 1, 4, 8, 19, 44, 103)&0.288457\\
181&(1, 1, 4, 8, 19, 45, 103)&0.288486\\
182&(1, 1, 4, 8, 19, 45, 104)&0.288513\\
183&(1, 1, 4, 8, 20, 45, 104)&0.288533\\
184&(1, 1, 4, 8, 20, 45, 105)&0.288563\\
185&(1, 1, 4, 8, 20, 46, 105)&0.288586\\
186&(1, 1, 4, 8, 20, 46, 106)&0.288617\\
187&(1, 1, 4, 8, 20, 46, 107)&0.288640\\
188&(1, 1, 4, 8, 20, 47, 107)&0.288661\\
189&(1, 1, 4, 8, 20, 47, 108)&0.288686\\
190&(1, 1, 4, 9, 20, 47, 108)&0.288711\\
191&(1, 1, 4, 9, 20, 47, 109)&0.288737\\
192&(1, 1, 4, 9, 20, 47, 110)&0.288756\\
193&(1, 1, 4, 9, 20, 48, 110)&0.288782\\
194&(1, 1, 4, 9, 20, 48, 111)&0.288803\\
195&(1, 1, 4, 9, 21, 48, 111)&0.288832\\
196&(1, 1, 4, 9, 21, 48, 112)&0.288854\\
197&(1, 1, 4, 9, 21, 49, 112)&0.288876\\
198&(1, 1, 4, 9, 21, 49, 113)&0.288900\\
199&(1, 1, 4, 9, 21, 49, 114)&0.288917\\
200&(1, 1, 4, 9, 22, 49, 114)&0.288937\\\hline
\end{array}
$$
\end{minipage}

\newpage

\begin{minipage}{0.5\textwidth}
$$
\begin{array}{rrr}
\hline
n & \qquad \quad\qquad (\ldots,b_2,b_1) & \qquad \quad f(n)\\ \hline
201&(1, 1, 4, 9, 22, 49, 115)&0.288956\\
202&(1, 1, 4, 9, 22, 50, 115)&0.288982\\
203&(1, 1, 4, 9, 22, 50, 116)&0.289003\\
204&(1, 1, 4, 9, 22, 51, 116)&0.289019\\
205&(1, 1, 4, 9, 22, 51, 117)&0.289041\\
206&(1, 1, 4, 10, 22, 51, 117)&0.289061\\
207&(1, 1, 4, 10, 22, 51, 118)&0.289084\\
208&(1, 1, 4, 10, 22, 51, 119)&0.289101\\
209&(1, 1, 4, 10, 22, 52, 119)&0.289122\\
210&(1, 1, 4, 10, 22, 52, 120)&0.289141\\
211&(1, 1, 4, 10, 23, 52, 120)&0.289160\\
212&(1, 1, 4, 10, 23, 52, 121)&0.289180\\
213&(1, 1, 4, 10, 23, 53, 121)&0.289197\\
214&(1, 1, 4, 10, 23, 53, 122)&0.289219\\
215&(1, 1, 4, 10, 23, 53, 123)&0.289234\\
216&(1, 2, 4, 10, 23, 53, 123)&0.289251\\
217&(1, 2, 4, 10, 23, 53, 124)&0.289268\\
218&(1, 2, 4, 10, 23, 54, 124)&0.289289\\
219&(1, 2, 4, 10, 23, 54, 125)&0.289307\\
220&(1, 2, 4, 10, 24, 54, 125)&0.289320\\
221&(1, 2, 4, 10, 24, 54, 126)&0.289340\\
222&(1, 2, 4, 10, 24, 55, 126)&0.289358\\
223&(1, 2, 4, 10, 24, 55, 127)&0.289379\\
224&(1, 2, 4, 10, 24, 55, 128)&0.289394\\
225&(1, 2, 4, 10, 24, 56, 128)&0.289411\\
226&(1, 2, 4, 10, 24, 56, 129)&0.289428\\
227&(1, 2, 4, 11, 24, 56, 129)&0.289441\\
228&(1, 2, 4, 11, 24, 56, 130)&0.289460\\
229&(1, 2, 4, 11, 24, 57, 130)&0.289473\\
230&(1, 2, 4, 11, 24, 57, 131)&0.289492\\
231&(1, 2, 4, 11, 24, 57, 132)&0.289507\\
232&(1, 2, 4, 11, 25, 57, 132)&0.289526\\
233&(1, 2, 4, 11, 25, 57, 133)&0.289541\\
234&(1, 2, 4, 11, 25, 58, 133)&0.289558\\
235&(1, 2, 4, 11, 25, 58, 134)&0.289575\\
236&(1, 2, 4, 11, 25, 58, 135)&0.289587\\
237&(1, 2, 4, 11, 25, 59, 135)&0.289602\\
238&(1, 2, 4, 11, 25, 59, 136)&0.289615\\
239&(1, 2, 4, 11, 26, 59, 136)&0.289633\\
240&(1, 2, 4, 11, 26, 59, 137)&0.289648\\
241&(1, 2, 4, 11, 26, 60, 137)&0.289661\\
242&(1, 2, 4, 11, 26, 60, 138)&0.289676\\
243&(1, 2, 5, 11, 26, 60, 138)&0.289693\\
244&(1, 2, 5, 11, 26, 60, 139)&0.289710\\
245&(1, 2, 5, 11, 26, 60, 140)&0.289722\\
246&(1, 2, 5, 11, 26, 61, 140)&0.289738\\
247&(1, 2, 5, 11, 26, 61, 141)&0.289751\\
248&(1, 2, 5, 11, 27, 61, 141)&0.289764\\
249&(1, 2, 5, 11, 27, 61, 142)&0.289778\\
250&(1, 2, 5, 11, 27, 62, 142)&0.289792\\ \hline
\end{array}
$$
\end{minipage}
\begin{minipage}{0.5\textwidth}
$$
\begin{array}{rrr}
\hline
n & \qquad \quad\qquad (\ldots,b_2,b_1) & \qquad \quad f(n)\\ \hline
251&(1, 2, 5, 11, 27, 62, 143)&0.289807\\
252&(1, 2, 5, 11, 27, 62, 144)&0.289818\\
253&(1, 2, 5, 12, 27, 62, 144)&0.289833\\
254&(1, 2, 5, 12, 27, 62, 145)&0.289845\\
255&(1, 2, 5, 12, 27, 63, 145)&0.289862\\
256&(1, 2, 5, 12, 27, 63, 146)&0.289875\\
257&(1, 2, 5, 12, 27, 64, 146)&0.289885\\
258&(1, 2, 5, 12, 27, 64, 147)&0.289899\\
259&(1, 2, 5, 12, 28, 64, 147)&0.289912\\
260&(1, 2, 5, 12, 28, 64, 148)&0.289927\\
261&(1, 2, 5, 12, 28, 64, 149)&0.289938\\
262&(1, 2, 5, 12, 28, 65, 149)&0.289951\\
263&(1, 2, 5, 12, 28, 65, 150)&0.289964\\
264&(1, 2, 5, 12, 29, 65, 150)&0.289972\\
265&(1, 2, 5, 12, 29, 65, 151)&0.289985\\
266&(1, 2, 5, 12, 29, 66, 151)&0.289996\\
267&(1, 2, 5, 12, 29, 66, 152)&0.290010\\
268&(1, 2, 5, 12, 29, 66, 153)&0.290020\\
269&(1, 2, 5, 13, 29, 66, 153)&0.290033\\
270&(1, 2, 5, 13, 29, 66, 154)&0.290044\\
271&(1, 2, 5, 13, 29, 67, 154)&0.290058\\
272&(1, 2, 5, 13, 29, 67, 155)&0.290070\\
273&(1, 2, 5, 13, 29, 67, 156)&0.290078\\
274&(1, 2, 5, 13, 29, 68, 156)&0.290091\\
275&(1, 2, 5, 13, 29, 68, 157)&0.290100\\
276&(1, 2, 5, 13, 30, 68, 157)&0.290113\\
277&(1, 2, 5, 13, 30, 68, 158)&0.290124\\
278&(1, 2, 5, 13, 30, 69, 158)&0.290134\\
279&(1, 2, 5, 13, 30, 69, 159)&0.290146\\
280&(1, 2, 6, 13, 30, 69, 159)&0.290158\\
281&(1, 2, 6, 13, 30, 69, 160)&0.290171\\
282&(1, 2, 6, 13, 30, 69, 161)&0.290179\\
283&(1, 2, 6, 13, 30, 70, 161)&0.290192\\
284&(1, 2, 6, 13, 30, 70, 162)&0.290202\\
285&(1, 2, 6, 13, 31, 70, 162)&0.290211\\
286&(1, 2, 6, 13, 31, 70, 163)&0.290222\\
287&(1, 2, 6, 13, 31, 71, 163)&0.290233\\
288&(1, 2, 6, 13, 31, 71, 164)&0.290245\\
289&(1, 2, 6, 13, 31, 71, 165)&0.290253\\
290&(1, 2, 6, 13, 31, 72, 165)&0.290263\\
291&(1, 2, 6, 13, 31, 72, 166)&0.290272\\
292&(1, 2, 6, 14, 31, 72, 166)&0.290284\\
293&(1, 2, 6, 14, 31, 72, 167)&0.290294\\
294&(1, 2, 6, 14, 31, 73, 167)&0.290302\\
295&(1, 2, 6, 14, 31, 73, 168)&0.290313\\
296&(1, 2, 6, 14, 32, 73, 168)&0.290322\\
297&(1, 2, 6, 14, 32, 73, 169)&0.290334\\
298&(1, 2, 6, 14, 32, 73, 170)&0.290342\\
299&(1, 2, 6, 14, 32, 74, 170)&0.290353\\
300&(1, 2, 6, 14, 32, 74, 171)&0.290362\\\hline
\end{array}
$$
\end{minipage}
}

\begin{figure}
\begin{center}
\includegraphics[scale=0.85]{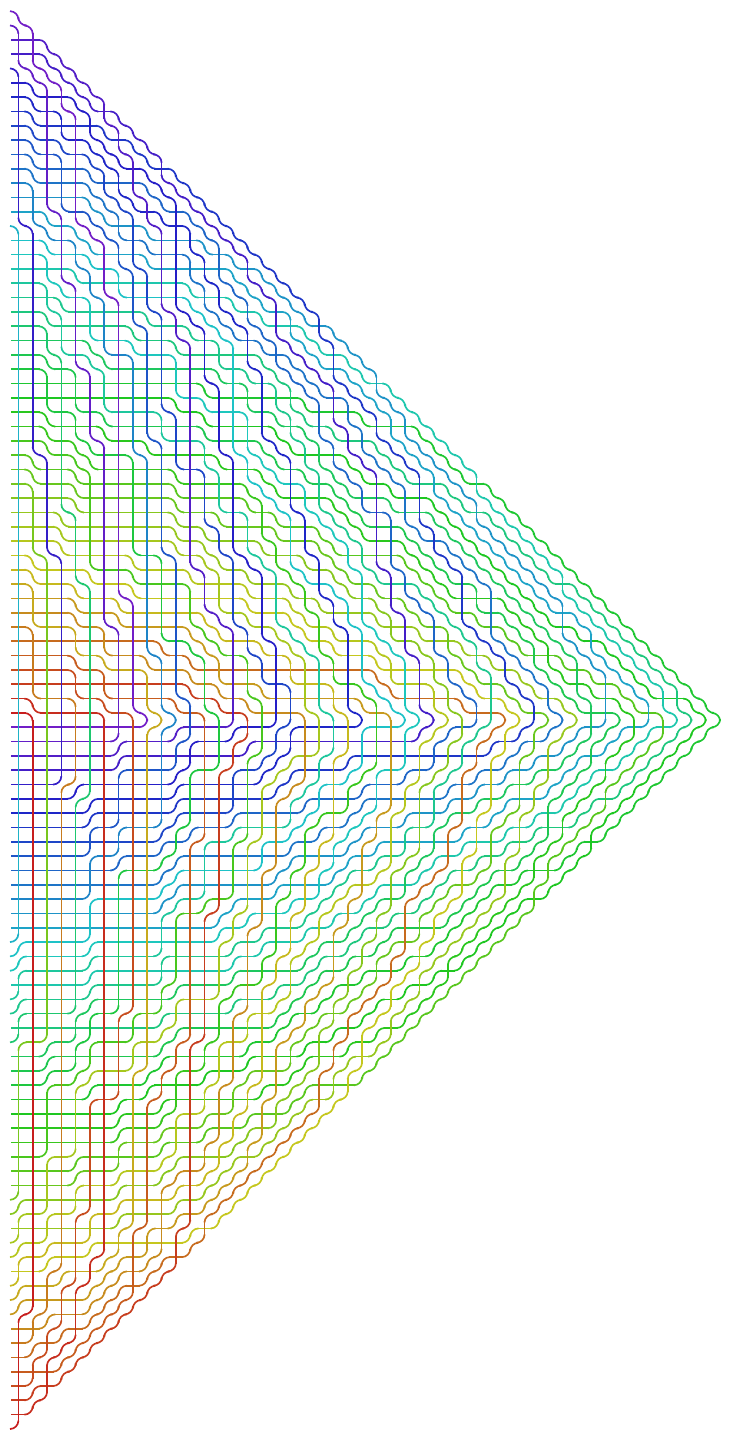}
\caption{Random double $rc$-graph corresponding to a permutation
in Figure~\ref{fig:random-permdrc}.}
\label{fig:random-permdrc-app}
\end{center}
\end{figure}

\end{document}